\documentclass{article}
\usepackage{amsmath}
\usepackage{amssymb}
\usepackage{caption}
\usepackage{subcaption}
\usepackage{graphicx}
\usepackage{mathtools}
\usepackage{arydshln}
\usepackage{amsthm}

\newtheorem{theorem}{Theorem}[section]
\newtheorem{corollary}{Corollary}[theorem]
\newtheorem{lemma}[theorem]{Lemma}
\newtheorem{proposition}[theorem]{Proposition}

\usepackage{xcolor}

\usepackage[utf8]{inputenc}

\title{On the structure of the Schur complement matrix for the Stokes equation.}
\author{Vladislav Pimanov, Ekaterina Muravleva, Ivan Oseledets, Oleg Iliev}
\date{\today}

\begin{document}

\maketitle

\begin{abstract}

\end{abstract}

\section{Introduction}

In this paper, we investigate the structure of the Schur complement matrix for the fully-staggered finite-difference discretization of the stationary Stokes equation. Specifically, we demonstrate that the structure of the Schur complement matrix depends qualitatively on a particular characteristic, namely the number of non-unit eigenvalues, and the two limiting cases are of special interest. Actually, the number of non-unit eigenvalues of the Schur complement matrix is determined by the boundary conditions imposed. In the paper, we consider enclosed Stokes flows, which are determined by the no-penetration boundary condition imposed on the normal component of the velocity. Depending on the type of the boundary condition imposed on the tangential component of the velocity, two boundary value problems (BVPs) are considered: the first is the (enclosed) Dirichlet problem, and the second is the (enclosed) Neumann problem. The Dirichlet problem is considered as the primary problem of interest, and the Neumann problem is considered as an auxiliary one.

In the auxiliary Neumann case, several neat properties hold. Firstly, the discrete Helmholtz-Hodge orthogonal decomposition holds for the discrete curl and discrete divergence operators. Secondly, the velocity vector Laplacian matrix is decomposed into a direct sum under this decomposition. These properties imply that the Schur complement matrix for the Neumann case, as well as its inverse, up to a constant factor in the nullspace, equals the identity matrix acting on the discrete pressure space. Thus, the Stokes problem is reduced to the unconstrained vector Laplacian equation for the velocity in this case.

The main observation is that the Dirichlet case can be considered as a perturbation of this degenerate Neumann case. Namely, we demonstrate that the Schur complement matrix assembled for the Dirichlet problem, as well as its inverse, can be considered as a rank-$r$ correction of the Schur complement matrix assembled for the Neumann problem. Moreover, the rank $r$ is determined by the number of the velocity nodes affected by the tangential Dirichlet boundary condition, hence the dimensionality of the problem can be reduced by one. The main result is related to the case when all the velocity nodes are affected by the Dirichlet condition. In this limiting case, a particularly simple structure of the Schur complement matrix is observed, which implies several practical outcomes.

The paper is organized as follows. In Section \ref{sec:bvps}, we formulate the enclosed Dirichlet and Neumann BVPs for the Stokes equation. In Section \ref{sec:discretization}, we introduce the fully-staggered finite-difference discretization and describe how the underlying discrete operators can be assembled using the Kronecker matrix product. In Section \ref{sec:schur}, we describe the structure of the Schur complement matrix. Namely, we provide explicit formulas for the Dirichlet Schur complement matrix written as a rank-$r$ correction of the Neumann Schur complement matrix. Finally, in Section \ref{sec:limit}, we present the limiting case when the correction matrix has the full rank.

\section{Enclosed Stokes flows: Dirichlet and Neumann BVPs.}\label{sec:bvps}
Let us consider the Stokes equation formulated in a bounded open domain $\Omega \subset \mathbb{R}^2$:
\begin{equation}\label{stokes}
    \begin{split}
        -\Delta \mathbf{u} + \nabla p  &= \mathbf{0}\phantom{0}  \text{ in } \Omega, \\
        -\nabla \cdot \mathbf{u} &= 0\phantom{\mathbf{0}}  \text{ in } \Omega, \\
    \end{split}
\end{equation}
where $\mathbf{u} = (u,v)^T$ denotes the fluid velocity and $p$ is the fluid pressure.  Let us further denote $\mathbf{n}$  the outward unit normal of the boundary $\partial \Omega$. Then, one can uniqely write the following decomposition for the velocity vector field $\mathbf{u}$:
\begin{equation*}
    \mathbf{u} = \mathbf{u}_{\perp} + \mathbf{u}_{\parallel},
\end{equation*}
 where $\mathbf{u}_{\perp} = (\mathbf{n} \cdot \mathbf{u})\mathbf{n}$ and $\mathbf{u}_{\parallel} = \mathbf{u} - \mathbf{u}_{\perp}$ are the normal and tangential components.
 In the present paper, we study two boundary value problems (BVPs) for the Stokes equation \eqref{stokes}. For both problems, we firstly impose the \textit{no-penetration} condition on $\partial \Omega$, which is defined as follows: 
\begin{equation}\label{normal_bc}
        \mathbf{u}_{\perp\phantom{\parallel}} = \mathbf{0}\phantom{0}  \mathrm{\; on} \; \partial \Omega.
\end{equation}
It should be noted, that the boundary condition \eqref{normal_bc} imposed on the whole boundary $\partial \Omega$ characterizes the class of \textit{enclosed flow} problems \cite{elman2014finite} and is sufficient for the pressure solution to be unique only up to a constant factor.
In addition to the constraint \eqref{normal_bc} imposed on the normal velocity component, also the tangential component of either the velocity itself or the normal traction $(\partial \mathbf{u} / \partial \mathbf{n} - p \mathbf{n})$ must be prescribed, which results in: 
\begin{subequations}\label{tangential_bc}
     \begin{align}
      \mathbf{u}_{\parallel}  &= \mathbf{f}\phantom{0} \mathrm{\; on} \; \partial \Omega, \label{dirichlet_bc} \\
      \partial \mathbf{u}_{\parallel} / \partial \mathbf{n} &= \mathbf{f}\phantom{0} \mathrm{\; on} \; \partial \Omega, \label{neumann_bc}
     \end{align}
\end{subequations}
respectively.
Further in the paper, we will refer the Stokes equation \eqref{stokes} with the boundary conditions \eqref{normal_bc}+\eqref{dirichlet_bc} and \eqref{normal_bc}+\eqref{neumann_bc} as the \textit{(enclosed) Dirichlet} and the \textit{(enclosed) Neumann} BVPs, respectively.
Note, that for the case $\mathbf{f} = \mathbf{0}$ the boundary conditions \eqref{normal_bc}+\eqref{dirichlet_bc} and \eqref{normal_bc}+\eqref{neumann_bc} are known as the \textit{no-slip} and \textit{free-slip} conditions, respectively. Also, recall that the well-known lid-driven cavity problem is a special case of the just described enclosed Dirichlet problem. 




\section{Finite-difference discretization on fully-staggered grids.}\label{sec:discretization}
In what follows, we consider $\Omega$ to be a unit two-dimensional domain with the following tensor-product structure:
\begin{equation}\label{Omega}
    \Omega = \omega \times \omega \subset \mathbb{R}^2,
\end{equation}
where $\omega$ denotes a unit one-dimensional interval:
\begin{equation}\label{omega}
    \omega = (0,1) \subset \mathbb{R}.
\end{equation}
Given the problem size $\mathrm{n}$ and the corresponding grid size $h = 1/\mathrm{n}$, we consider the classical fully-staggered finite difference scheme \cite{harlow1965numerical}, for which the discretized velocities $\mathbf{u}_h = (u_h, v_h)^T$, the discretized pressure $p_h$, and the discretized velocity curl, denoted $q_h$, live on different grids. 
Namely, we discretize $\Omega$ from \eqref{Omega} with four different tensor-product grids:
\begin{equation}\label{Omega_discrete}
\begin{split}
    \Omega_h^{u} = \overline{\omega}_h \times {\omega}_h, \quad
    \Omega_h^{v} =  {\omega}_h \times \overline{\omega}_h, \\
    \Omega_h^{p} =  \overline{\omega}_h \times \overline{\omega}_h, \quad
    \Omega_h^{q} =  {\omega}_h \times {\omega}_h, \\
\end{split}
\end{equation}
where ${\omega}_h$ and $\overline{\omega}_h$ denote the aligned grid and the shifted (by $h/2$) grid, which are discretizations of $\omega$ from \eqref{omega}, given as follows:
\begin{equation}\label{omega_discrete}
    \omega_h = 
        (h, 1 - h; \; h), \quad
    \overline{\omega}_h = 
        (h/2, 1 - h/2; \; h).   
\end{equation}
For an illustration of the fully-staggered grids \eqref{Omega_discrete} for $\mathrm{n}=4$, see Fig. \ref{fig:grid_cavity}.
It is worth noting, that such fully-staggered discretization is known to be \textit{structure-preserving} in the sense that many fundamental structures of the continuous model, e.g. mass and momentum conservation laws, are preserved at the discrete level.
\begin{figure}
     \centering
         \includegraphics[width=.5\textwidth]{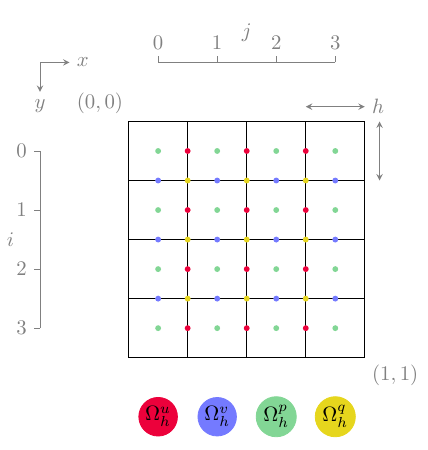}
         \caption{Fully-staggered grids for $\mathrm{n}=4$.}
         \label{fig:grid_cavity}
\end{figure}
Formally, we have the discrete functions belonging to different discrete spaces:
\begin{equation}\label{discrete_spaces}
\begin{split}
        \mathcal{U}_{h} &= \{u_h : \Omega_h^{u} \to \mathbb{R}\}, \quad
        \mathcal{V}_{h} = \{v_h : \Omega_h^{v} \to \mathbb{R}\}, \\
        \mathcal{P}_{h} &= \{p_h : \Omega_h^{p} \to \mathbb{R}\}, \quad
        \mathcal{Q}_{h} = \{q_h : \Omega_h^{q} \to \mathbb{R}\}. \\
\end{split}
\end{equation}
We identify these discrete spaces $\mathcal{U}_{h}$, $\mathcal{V}_{h}$, $\mathcal{P}_{h}$, and $\mathcal{Q}_{h}$ with the vector spaces $\mathbb{R}^{\mathrm{n}(\mathrm{n}-1)}$, $\mathbb{R}^{(\mathrm{n}-1)\mathrm{n}}$, $\mathbb{R}^{\mathrm{n}\mathrm{n}}$, and $\mathbb{R}^{(\mathrm{n}-1)(\mathrm{n}-1)}$, respectively. Then, the underlying two-dimensional discrete operators can be assembled from the one-dimensional operators  using the Kronecker matrix product (see, e.g., \cite{hackbusch2006low}).
For example, the identity operators $I^u:\mathcal{U}_h \to \mathcal{U}_h$, $I^v:\mathcal{V}_h \to \mathcal{V}_h$, $I^p:\mathcal{P}_h \to \mathcal{P}_h$, $I^q:\mathcal{Q}_h \to \mathcal{Q}_h$ are constructed using the Kronecker product as follows:
\begin{equation}\label{kron_identities}
    \begin{split}
        I^u &= I^{\overline{\omega}_h} \otimes I^{{\omega}_h}, \quad I^u \in \mathbb{R}^{\mathrm{n}(\mathrm{n}-1) \times \mathrm{n} (\mathrm{n}-1)}, \\
        I^v &= I^{{\omega}_h} \otimes I^{\overline{\omega}_h}, \quad I^v \in \mathbb{R}^{(\mathrm{n}-1)\mathrm{n} \times (\mathrm{n}-1) \mathrm{n}}, \\
        I^p &= I^{\overline{\omega}_h} \otimes I^{\overline{\omega}_h}, \quad I^p \in \mathbb{R}^{\mathrm{n}\mathrm{n} \times \mathrm{n} \mathrm{n}}, \\
        I^q &= I^{{\omega}_h} \otimes I^{{\omega}_h}, \quad I^q \in \mathbb{R}^{(\mathrm{n}-1)(\mathrm{n}-1) \times (\mathrm{n}-1) (\mathrm{n}-1)}, \\
    \end{split}
\end{equation}
where $I^{{\omega}_h}$ and $I^{\overline{\omega}_h}$ are the one-dimensional identity operators corresponding to the aligned grid $\omega_h$ and the shifted grid $\overline{\omega}_h$ from \eqref{omega_discrete}, respectively:
\begin{equation*}
    I^{{\omega}_h} = I^{\mathrm{n} - 1} \in \mathbb{R}^{(\mathrm{n} - 1) \times (\mathrm{n} - 1)}, \quad I^{\overline{\omega}_h} = I^{\mathrm{n}} \in \mathbb{R}^{\mathrm{n} \times \mathrm{n}}.
\end{equation*}
Finally, the velocity vector identity operator is composed as follows:
\begin{equation}\label{velocity_identity}
    {\mathbf{I^u}} = 
    \begin{bmatrix}
    I^{{u}} & \\
    & I^{{v}}
    \end{bmatrix}
    , \quad \mathbf{I^u}:(\mathcal{U}_h, \mathcal{V}_h)^T \to (\mathcal{U}_h, \mathcal{V}_h)^T.
\end{equation}
After finite-difference discretization of the formulated in Section \ref{sec:bvps} enclosed Dirichlet and Neumann BVPs on fully-staggered grids \eqref{Omega_discrete}, we obtain a block system of the following saddle-point structure:
\begin{equation}\label{coupled_matrix}
\begin{bmatrix}
    \mathbf{A} & \mathbf{B}^T \\ 
    \mathbf{B} &
\end{bmatrix}
    \begin{bmatrix}
    \mathbf{u}_h \\ 
    p_h
    \end{bmatrix}
    =
    \begin{bmatrix}
    \mathbf{f}_h \\ 
     0
    \end{bmatrix}, 
\end{equation}
where $\mathbf{f}_h = (f^u_h, f^v_h)^T \in (\mathcal{U}_h, \mathcal{V}_h)^T$ is a discretization of the boundary term $\mathbf{f}$ from \eqref{tangential_bc}. The matrices $\mathbf{B}$, $\mathbf{B}^T$, and $\mathbf{A}$ represent the discrete counterparts of the velocity (negative) divergence, the pressure gradient, and the velocity vector Laplacian operators from the Stokes equation \eqref{stokes}, respectively. The boundary conditions \eqref{normal_bc} and \eqref{tangential_bc} are also incorporated into these matrices. It should be noted, that the matrix $\mathbf{B}$ depends solely on the normal boundary condition \eqref{normal_bc}, which makes it identical for both the Dirichlet and Neumann BVPs. However, for the Laplacian matrix we have $\mathbf{A} = \mathbf{A_D}$ or $\mathbf{A} = \mathbf{A_N}$ depending on whether the condition \eqref{dirichlet_bc} or \eqref{neumann_bc} is imposed on the tangential velocity component.

\subsection{Assembling discrete operators for the Neumann BVP.}\label{sec:neumann_discretization}

The discrete operators for the enclosed Neumann problem have a simple structure such that the block matrices $\mathbf{A}$ and $\mathbf{B}$ from \eqref{coupled_matrix} can be assembled using the Kronecker matrix product from a single matrix $\mathrm{B}:(\omega_h\to\mathbb{R}) \to (\overline{\omega}_h\to\mathbb{R})$, which is a discretization of the one-dimensional derivative operator, defined as follows:
\begin{equation*}
\mathrm{B} = \frac{1}{h}
    \begin{pmatrix}
    \phantom{-}1 &     &      &  \\ 
        -1 &  \phantom{-}1   &      &  \\ 
        &   \ddots  & \ddots &     \\
     &     &    -1  & \phantom{-}1 \\ 
     &            &   & -1 \\
    \end{pmatrix}
    ,  \quad \mathrm{B} \in \mathbb{R}^{\mathrm{n} \times (\mathrm{n} - 1)}.
\end{equation*}
Namely, the discrete co-directional derivatives of the velocity, denoted $\mathrm{B}_\mathrm{x}^u:\mathcal{U}_h\to\mathcal{P}_h$ and $\mathrm{B}_\mathrm{y}^v:\mathcal{V}_h\to\mathcal{P}_h$, are assembled as follows:
\begin{equation*}
\begin{split}
        \mathrm{B}_\mathrm{x}^u = I^{{\overline{\omega}_h}} \otimes \mathrm{B},\quad \mathrm{B}_\mathrm{x}^u \in \mathbb{R}^{\mathrm{n}\mathrm{n} \times \mathrm{n}(\mathrm{n}-1)}, \\
        \mathrm{B}_\mathrm{y}^v = \mathrm{B} \otimes I^{{\overline{\omega}_h}}, \quad \mathrm{B}_\mathrm{y}^v \in \mathbb{R}^{\mathrm{n}\mathrm{n} \times (\mathrm{n}-1)\mathrm{n}},
\end{split}
\end{equation*}
and the discrete derivatives of the velocity curl, denoted $\mathrm{B}_\mathrm{x}^q: \mathcal{Q}_h \to \mathcal{V}_h$ and $\mathrm{B}_\mathrm{y}^q: \mathcal{Q}_h \to \mathcal{U}_h$, are assembled as follows:
\begin{equation*}
    \begin{split}
    \mathrm{B}_\mathrm{x}^q = I^{{\omega_h}} \otimes \mathrm{B}, \quad \mathrm{B}_\mathrm{x}^q \in \mathbb{R}^{(\mathrm{n}-1)\mathrm{n}\times (\mathrm{n} - 1) (\mathrm{n}-1)}, \\
    \mathrm{B}_\mathrm{y}^q = \mathrm{B} \otimes I^{{\omega_h}}, \quad \mathrm{B}_\mathrm{y}^q \in \mathbb{R}^{\mathrm{n}(\mathrm{n}-1) \times (\mathrm{n} - 1) (\mathrm{n} - 1)}.    
    \end{split}
\end{equation*}
Then, the discrete (negative) divergence of the velocity is composed as follows:
\begin{equation}\label{divergence}
    \mathbf{B} = 
    \begin{bmatrix}
    -\mathrm{B}_\mathrm{x}^u  & -\mathrm{B}_\mathrm{y}^v
    \end{bmatrix}, \quad 
    \mathbf{B}:(\mathcal{U}_h,\mathcal{V}_h)^T\to\mathcal{P}_h,
\end{equation}
and the discrete curl of the velocity curl is composed as follows:
\begin{equation*}\label{curl}
    \mathbf{C}^T = 
    \begin{bmatrix}
    -\mathrm{B}_\mathrm{y}^q  \\
    \phantom{-}\mathrm{B}_\mathrm{x}^q
    \end{bmatrix}, \quad 
    \mathbf{C}^T:\mathcal{Q}_h\to(\mathcal{U}_h,\mathcal{V}_h)^T.
\end{equation*}
Next, for the discrete gradient of the pressure, we have:
\begin{equation}\label{grad_div_relation}
\begin{bmatrix}
    \mathrm{B}^p_\mathrm{x} \\ 
    \mathrm{B}^p_\mathrm{y}
\end{bmatrix}
=
\begin{bmatrix}
    -{\mathrm{B}^u_\mathrm{x}}^T \\
    -{\mathrm{B}^v_\mathrm{y}}^T
\end{bmatrix}
= \mathbf{B}^T,
\quad \mathbf{B}^T:\mathcal{P}_h\to(\mathcal{U}_h,\mathcal{V}_h)^T,
\end{equation}
and for the discrete curl of the velocity, we have:
\begin{equation}\label{curl_curl_relation}
\begin{bmatrix}
    \mathrm{B}^u_\mathrm{y} &
    -\mathrm{B}^v_\mathrm{x}
\end{bmatrix}
= 
\begin{bmatrix}
    -{\mathrm{B}^q_\mathrm{y}}^T &
    {\mathrm{B}^q_\mathrm{x}}^T
\end{bmatrix}
= {\mathbf{C}^T}^T = \mathbf{C}
, \quad \mathbf{C}:(\mathcal{U}_h,\mathcal{V}_h)^T\to\mathcal{Q}_h.
\end{equation}
It is worth noting, that the relation \eqref{grad_div_relation} reflects the fact that the pressure gradient operator, by definition, is conjugate to the negative divergence velocity operator. Thus, the following property is always satisfied:
\begin{lemma}\label{div_curl_lemma}
\begin{equation}\label{div_curl_0}
    \begin{bmatrix}
    \mathrm{B}_\mathrm{x}^u  & \mathrm{B}_\mathrm{y}^v
    \end{bmatrix}
    \begin{bmatrix}
    -\mathrm{B}_\mathrm{y}^q  \\
    \phantom{-}\mathrm{B}_\mathrm{x}^q 
    \end{bmatrix}
    = {0},
\end{equation}
which preserves the continuous property:
\begin{equation*}
    \nabla \cdot (\nabla \times q) = 0.
\end{equation*}
\end{lemma}
\begin{proof}
\begin{equation*}
\begin{split}
    \mathbf{B}\mathbf{C}^T =  
    &-\mathrm{B}_\mathrm{x}^u \mathrm{B}_\mathrm{y}^q + \mathrm{B}_\mathrm{y}^v\mathrm{B}_\mathrm{x}^q  = \\
    &-(I^{{\overline{\omega}_h}} \otimes \mathrm{B}) 
    (\mathrm{B} \otimes I^{{\omega_h}}) + 
    (\mathrm{B} \otimes I^{{\overline{\omega}_h}})
    (I^{{\omega_h}} \otimes \mathrm{B}) = \\
    &-\mathrm{B} \otimes \mathrm{B} + \mathrm{B} \otimes \mathrm{B} = 0.
\end{split}
\end{equation*}
\end{proof}

\noindent
However, the relation \eqref{curl_curl_relation} is not valid in general, but only for certain "energy-conserving" formulations. In fact, it is the fulfillment of the relation \eqref{curl_curl_relation} that entails the discrete Helmholtz-Hodge decomposition for the enclosed Neumann problem, which is discussed in Section \ref{sec:schur_neumann}.
So, the following relation is satisfied:
\begin{lemma}\label{curl_grad_lemma}
\begin{equation}\label{div_curl_0}
\begin{bmatrix}
    \mathrm{B}^u_\mathrm{y} &
    -\mathrm{B}^v_\mathrm{x}
\end{bmatrix}
\begin{bmatrix}
    \mathrm{B}^p_\mathrm{x} \\ 
    \mathrm{B}^p_\mathrm{y}
\end{bmatrix}
= {0},
\end{equation}
which preserves the continuous property:
\begin{equation*}
    \nabla \times (\nabla  {p}) = {0}.
\end{equation*}
\end{lemma}
\begin{proof}
\begin{equation*}
    \mathbf{C}\mathbf{B}^T = (\mathbf{B}\mathbf{C}^T)^T = 0^T = 0.
\end{equation*}
\end{proof}

\noindent
We also state the following equalities:
\begin{lemma}
\begin{equation*}
    \mathrm{B}^p_\mathrm{y}\mathrm{B}^u_\mathrm{x} = \mathrm{B}^q_\mathrm{x}\mathrm{B}^u_\mathrm{y}, \quad
    \mathrm{B}^p_\mathrm{x}\mathrm{B}^v_\mathrm{y} =
    \mathrm{B}^q_\mathrm{y}\mathrm{B}^v_\mathrm{x}.
\end{equation*}
which preserve the continuous relations:
\begin{equation*}
    u_{xy} = u_{yx}, \quad v_{yx} = v_{xy}.
\end{equation*}
\end{lemma}
\begin{proof}
\begin{equation*}
    \begin{split}
        -\mathrm{B}^p_\mathrm{y}\mathrm{B}^u_\mathrm{x} = 
        {\mathrm{B}^v_\mathrm{y}}^T\mathrm{B}^u_\mathrm{x} = 
        (\mathrm{B} \otimes I^{{\overline{\omega}_h}})^T(I^{{\overline{\omega}_h}} \otimes \mathrm{B}) &= \\
        =(\mathrm{B}^T \otimes I^{{\overline{\omega}_h}})(I^{{\overline{\omega}_h}} \otimes \mathrm{B}) = 
        \mathrm{B}^T &\otimes \mathrm{B} = 
        (I^{{\omega_h}} \otimes \mathrm{B})(\mathrm{B}^T \otimes I^{{\omega_h}}) = \\
        &= (I^{{\omega_h}} \otimes \mathrm{B})(\mathrm{B} \otimes I^{{\omega_h}})^T =
        \mathrm{B}^q_\mathrm{x}{\mathrm{B}^q_\mathrm{y}}^T =
        -\mathrm{B}^q_\mathrm{x}\mathrm{B}^u_\mathrm{y}.
    \end{split}
\end{equation*}
\end{proof}

\noindent
Finally, the discrete (negative) velocity Laplacian is assembled as follows:
\begin{proposition}
\begin{equation}\label{laplacian_neumann}
        \mathbf{A_N} \; = \; \mathbf{B}^T\mathbf{B} \; + \; \mathbf{C}^T\mathbf{C}, \quad
        \mathbf{A_N}:(\mathcal{U}_h,\mathcal{V}_h)^T\to(\mathcal{U}_h,\mathcal{V}_h)^T.
\end{equation}
which preserves the continuous identity:
\begin{equation*}
    -\Delta \; = \; -\nabla \nabla \cdot \; + \; \nabla \times \nabla \times.
\end{equation*}
\end{proposition}
\begin{proof}
\begin{equation*}
\begin{split}
\mathbf{B}^T\mathbf{B} + \mathbf{C}^T\mathbf{C} &= 
\\ &=
\begin{bmatrix}
-\mathrm{B}_\mathrm{x}^p\mathrm{B}_\mathrm{x}^u & -\mathrm{B}_\mathrm{x}^p\mathrm{B}_\mathrm{y}^v \\ 
-\mathrm{B}_\mathrm{y}^p\mathrm{B}_\mathrm{x}^u & -\mathrm{B}_\mathrm{y}^p\mathrm{B}_\mathrm{y}^v 
\end{bmatrix}
+
\begin{bmatrix}
-\mathrm{B}_\mathrm{y}^q\mathrm{B}_\mathrm{y}^u & \phantom{-}\mathrm{B}_\mathrm{y}^q\mathrm{B}_\mathrm{x}^v \\ 
\phantom{-}
\mathrm{B}_\mathrm{x}^q\mathrm{B}_\mathrm{y}^u & -\mathrm{B}_\mathrm{x}^q\mathrm{B}_\mathrm{x}^v
\end{bmatrix} = \\
&=
\begin{bmatrix}
    -\mathrm{B}_\mathrm{x}^p \mathrm{B}_\mathrm{x}^u - \mathrm{B}_\mathrm{y}^q \mathrm{B}_\mathrm{y}^u &  \\ 
     & -\mathrm{B}_\mathrm{x}^q \mathrm{B}_\mathrm{x}^v - \mathrm{B}_\mathrm{y}^p \mathrm{B}_\mathrm{y}^v
\end{bmatrix} 
= 
\begin{bmatrix}
    \mathrm{A}_N^u &  \\
    & \mathrm{A}_N^v
\end{bmatrix}
= \mathbf{A_N}.
\end{split}
\end{equation*}
\end{proof}

\subsection{Assembling discrete operators for the Dirichlet BVP.}\label{sec:dirichlet_discretization}
As mentioned earlier, the matrix $\mathbf{B}$ defined in \eqref{divergence} is the same for both Dirichlet and Neumann problems. As for the matrix $\mathbf{A}$, we have that the discrete Dirichlet velocity Laplacian operator is assembled as a diagonal perturbation of the discrete Neumann velocity Laplacian operator $\mathbf{A_N}$ defined in \eqref{laplacian_neumann}:
\begin{equation}\label{laplacian_dirichlet}
    \mathbf{A_D} = \mathbf{A_N} + \dfrac{2}{h^2}{\mathbf{I^u_{\sim}}}, \quad
    \mathbf{A_D}:(\mathcal{U}_h, \mathcal{V}_h)^T \to (\mathcal{U}_h, \mathcal{V}_h)^T,
\end{equation}
where the diagonal perturbation matrix is given as follows:
\begin{equation}\label{perturbation_matrix}
\begin{split}
    {\mathbf{I^u_{\sim}}} &= 
    \begin{bmatrix}
        I_{\sim}^u & \\
         & I_{\sim}^v
    \end{bmatrix}
    = 
    \begin{bmatrix}
         I^{{\overline{\omega}_h}}_{\sim}  \otimes I^{\omega_h} & \\
         & I^{\omega_h} \otimes I^{{\overline{\omega}_h}}_{\sim}
    \end{bmatrix}
    , \quad 
    {\mathbf{I^u_{\sim}}} : (\mathcal{U}_h, \mathcal{V}_h)^T \to (\mathcal{U}_h, \mathcal{V}_h)^T,
    \end{split}
\end{equation}
where $I^{\omega_h}$ is defined in \eqref{omega_discrete}, and $I^{{\overline{\omega}_h}}_{\sim} : ({\overline{\omega}_h} \to \mathbb{R}) \to ({\overline{\omega}_h} \to \mathbb{R})$ is defined as follows:
\begin{equation*}
I^{{\overline{\omega}_h}}_{\sim} = 
    \begin{pmatrix}
    1 & & & & \\
    & 0 & & &\\
    & & \ddots & & \\
    & & & 0 & \\
    & & & & 1
    \end{pmatrix}
    , \quad I^{{\overline{\omega}_h}}_{\sim} \in \mathbb{R}^{\mathrm{n} \times \mathrm{n}}.
\end{equation*}
It should be noted, that the perturbation matrix $\mathbf{I^u_{\sim}}$ has only $r = \mathrm{rank}(\mathbf{I^u_{\sim}})$ non-zeros on the diagonal which correspond to the velocity nodes affected by the tangential boundary condition \eqref{dirichlet_bc} imposed on $\partial \Omega$, namely, we have:
\begin{equation}\label{rank}
    r = 4(\mathrm{n}-1) = \mathcal{O}(\mathrm{n}).
\end{equation}


\section{Structure of the Schur complement matrix}\label{sec:schur}

Let us consider the Schur complement reduction of the coupled system \eqref{coupled_matrix} to the equivalent equation on the discrete pressure:
\begin{equation}\label{schur_system}
    S p_h = g_h,
\end{equation}
where $S:\mathcal{P}_h\to\mathcal{P}_h$ denotes the Schur complement matrix, and $g_h\in\mathcal{P}_h$ is the right-hand-side in the reduced equation, which are defined as follows:
\begin{equation*}\label{schur_rhs}
     S = \mathbf{B}\mathbf{A}^{-1} \mathbf{B}^T, \quad g_h = \mathbf{B}\mathbf{A}^{-1}\mathbf{f}_h.
\end{equation*}
Once the pressure is computed, the velocity can be recovered by solving the following equation:
\begin{equation*}
\label{velocity_schur}
    \mathbf{A}\mathbf{u}_h = \mathbf{f}_h - \mathbf{B}^Tp_h.
\end{equation*}
In the present Section, we investigate the structure of the Schur complement matrices $S = S_D = \mathbf{B}\mathbf{A}_\mathbf{D}^{-1}\mathbf{B}^T$ and $S = S_N = \mathbf{B}\mathbf{A}_\mathbf{N}^{-1}\mathbf{B}^T$ computed for the velocity Laplacian matrices $\mathbf{A_D}$ and $\mathbf{A_N}$ defined in \eqref{laplacian_dirichlet} and \eqref{laplacian_neumann} for the enclosed Dirichlet and Neumann BVPs formulated in Section \ref{sec:bvps}, respectively.

\subsection{Structure of $S_N$.}\label{sec:schur_neumann}

In this subsection, we demonstrate that the Neumann Schur complement matrix $S_N = \mathbf{B}\mathbf{A}_\mathbf{N}^{-1}\mathbf{B}^T$ is reduced to the pressure identity operator $I^p$ defined in \eqref{kron_identities}, up to a one-dimensional constant nullspace.
To show this, we exploit the Helmholtz-Hodge decomposition, which in the continuous case states that any vector field can be uniquely represented as the sum of a non-divergent field and a non-rotating field.
It turns out, that for the discrete divergence operator $\mathbf{B}$ defined in \eqref{divergence} and the discrete curl operator $\mathbf{C}$ defined in \eqref{curl_curl_relation}, the discrete Helmholtz-Hodge decomposition is preserved, which can be formulated as follows:
\begin{theorem}\label{theorem:dhhd}
\begin{equation}\label{helmholtz_decomposition}
    (\mathcal{U}_h, \mathcal{V}_h)^T = \mathrm{Ker} \mathbf{B} \oplus \mathrm{Ker} \mathbf{C},
\end{equation}
and 
 \begin{equation}\label{dimension_relation}
          \mathrm{dim}(\mathrm{Ker}\mathbf{B}) = \mathrm{n}^2 - 1, \quad
          \mathrm{dim}(\mathrm{Ker}\mathbf{C}) = (\mathrm{n} - 1)^2.
 \end{equation}
\end{theorem}
\begin{proof}
Firstly, from Lemma \ref{div_curl_lemma} or \ref{curl_grad_lemma}, we have:
\begin{equation*}
\mathrm{Im} \mathbf{C}^T \subset \mathrm{Ker} \mathbf{B},
\end{equation*}
which implies:
\begin{equation}\label{subset}
(\mathrm{Ker} \mathbf{C})^\perp \subset \mathrm{Ker} \mathbf{B},
\end{equation}
since $(\mathrm{Ker} \mathbf{C})^\perp = \mathrm{Im} \mathbf{C}^T$.
Secondly, let us take $\mathbf{x} \in (\mathcal{U}_h, \mathcal{V}_h)^T$ such that $\mathbf{x} \in \mathrm{Ker}\mathbf{B} \cap \mathrm{Ker}\mathbf{C}$, then we have $\mathbf{x} \in \mathrm{Ker}(\mathbf{B}^T\mathbf{B}) \cap \mathrm{Ker}(\mathbf{C}^T\mathbf{C})$. Thus, using the representation \eqref{laplacian_neumann}, we have:
\begin{equation*}
(\mathbf{B}^T\mathbf{B}) \mathbf{x} + (\mathbf{C}^T\mathbf{C}) \mathbf{x} = \mathbf{A_N} \mathbf{x} = \mathbf{0},
\end{equation*}
which implies that $\mathbf{x} = \mathbf{0}$, since the matrix $\mathbf{A_N}$ has full rank. Therefore, we have shown that $\mathrm{Ker}\mathbf{B} \cap \mathrm{Ker}\mathbf{C} = {\mathbf{0}}$, which completes the proof of the decomposition \eqref{helmholtz_decomposition}, taking into account the relation \eqref{subset}.
Next, in order to prove the dimension relations \eqref{dimension_relation}, we note that the fully-staggered finite-difference discretization of the pressure gradient operator $\mathbf{B}^T$ defined in \eqref{grad_div_relation} has a one-dimensional nullspace spanned by the constant vector $\mathbf{1}$, which is defined as follows:
\begin{equation}\label{pressure_constant}
\mathbf{1} = h(1,\ldots,1)^T \in \mathcal{P}_h, \quad |\mathbf{1}|_2 = 1.
\end{equation}
It is also worth mentioning that $\mathrm{dim}(\mathrm{Ker}\mathbf{B}) = \mathrm{rank}(\mathbf{B}^T\mathbf{B})$ and $\mathrm{dim}(\mathrm{Ker}\mathbf{C}) = \mathrm{rank}(\mathbf{C}^T\mathbf{C})$.
\end{proof}

\noindent
Several important conclusions can be drawn from Theorem \ref{theorem:dhhd}.
First, the discrete Neumann velocity Laplacian operator $\mathbf{A_N}$ defined in \eqref{laplacian_neumann} expands into the direct sum with respect to the discrete Helmholtz-Hodge decomposition \eqref{helmholtz_decomposition}. Namely, for the inverse we have:
\begin{corollary}
\begin{equation}\label{neumann_laplace_dhhd}
    \mathbf{A}_\mathbf{N}^{-1} = (\mathbf{B}^T\mathbf{B})^\dagger + (\mathbf{C}^T\mathbf{C})^\dagger,
\end{equation}
where the superscript $\dagger$ denotes the Moore-Penrose pseudoinverse.
\end{corollary}

\noindent
Secondly, for the Neumann Schur complement matrix we have the following explicit representation:
\begin{corollary}
\begin{equation}\label{sn}
    S_N = I^p - (\mathbf{1} \cdot \mathbf{1}^T),
\end{equation}
where $\mathbf{1} \in \mathcal{P}^h$ is the pressure constant defined in \eqref{pressure_constant}.
\end{corollary}

\begin{proof}
Using decomposition \eqref{neumann_laplace_dhhd}, we have:
\begin{equation*}
\begin{split}
        S_N = \mathbf{B}\mathbf{A}_\mathbf{N}^{-1}\mathbf{B}^T = \mathbf{B}((\mathbf{B}^T\mathbf{B})^\dagger + (\mathbf{C}^T\mathbf{C})^\dagger)\mathbf{B}^T = \\
        = \mathbf{B}(\mathbf{B}^T\mathbf{B})^\dagger\mathbf{B}^T = \mathbf{B}\mathbf{B}^\dagger (\mathbf{B}\mathbf{B}^\dagger)^T = \\
    = (\mathbf{B}\mathbf{B}^\dagger)^2 = \mathbf{B}\mathbf{B}^\dagger,
    \end{split}
\end{equation*}
where the resulting matrix $\mathbf{B}\mathbf{B}^\dagger$ is, by definition, an orthogonal projector onto the kernel of the discrete gradient operator $\mathbf{B}^T$, which completes the proof.
\end{proof}

\noindent
Finally, for the inverse of $S_N$, by definition, we have:
\begin{corollary}
\begin{equation}\label{inverse_sn}
    S_N^\dagger = S_N^{\phantom{\dagger}}, \quad
    S_N^{\dagger} : \mathcal{P}_h \to \mathcal{P}_h.
\end{equation}
\end{corollary}
\noindent
Thus, it has been shown that the Neumann Schur complement matrix $S_N$, as well as its inverse $S_N^\dagger $, equals the identity operator $I^p$ acting on the discrete pressure space $\mathcal{P}_h$, up to a one-dimensional kernel of the discrete pressure gradient operator $\mathbf{B}^T$.
\subsection{Structure of $S_D$.}\label{sec:schur_dirichlet}

Let us consider the following decomposition of the velocity perturbation matrix ${\mathbf{I^u_{\sim}}}$ defined in \eqref{perturbation_matrix}:
\begin{equation}\label{lr_pert}
    {\mathbf{I^u_{\sim}}} = \mathbf{U}^T\mathbf{U},
\end{equation}
where $\mathbf{U} \in \mathbb{R}^{r \times (\mathrm{n}(\mathrm{n}-1) + (\mathrm{n}-1)\mathrm{n})}$ is an operator which extracts the velocity nodes affected by the tangential boundary condition, and the rank $r$ is defined in \eqref{rank}. Then, the matrix $S_D = \mathbf{B}\mathbf{A}_\mathbf{D}^{-1}\mathbf{B}^T$, as well as its inverse $S_D^{\dagger}$, can be written as rank-$r$ perturbations of the identity matrices $S_N$ and $S_N^{\dagger}$ defined in \eqref{sn} and \eqref{inverse_sn}, respectively:
\begin{theorem}
\begin{subequations}
\begin{align}
        S_D &= S_N - (\mathbf{U} \mathbf{B}^\dagger)^T (K_1)^{-1} (\mathbf{U} \mathbf{B}^\dagger), \label{sd_lr}\\
        S_D^{\dagger} &= S_N^{\dagger} + (\mathbf{U} \mathbf{B}^\dagger)^T (K_2)^{-1} (\mathbf{U} \mathbf{B}^\dagger), \label{sd_lr_inv}
\end{align}
\end{subequations}
where $r$-dimensional kernels $K_1, K_2 \in \mathbb{R}^{r \times r}$ are defined as follows:
\begin{equation*}
    \begin{split}
        K_1 &= \mathbf{U} (\frac{h^2}{2}{I^\mathbf{u}} +  (\mathbf{B}^{T}\mathbf{B})^{\dagger} +  (\mathbf{C}^{T}\mathbf{C})^{\dagger}) \mathbf{U}^T, \\ 
        K_2 &= \mathbf{U} ({\frac{h^2}{2}I^\mathbf{u}} + (\mathbf{C}^{T}\mathbf{C})^{\dagger}) \mathbf{U}^T.
    \end{split}
\end{equation*}
\end{theorem}

\begin{proof}
Taking into account the representation formula \eqref{laplacian_dirichlet} and the decomposition \eqref{lr_pert}, the Sherman–Morrison–Woodbury formula for the inverse of $\mathbf{A_D}$ gives:
    \begin{equation*}\label{ad_inv}
    \begin{split}
        \mathbf{A}_\mathbf{D}^{-1} &= \mathbf{A}_\mathbf{N}^{-1} - \mathbf{A}_\mathbf{N}^{-1}\mathbf{U}^T({\frac{h^2}{2}I^r} + \mathbf{U}\mathbf{A_N}\mathbf{U}^T)^{-1}\mathbf{U}\mathbf{A}_\mathbf{N}^{-1} = \\
        &= \mathbf{A}_\mathbf{N}^{-1} - (\mathbf{U}\mathbf{A}_\mathbf{N}^{-1})^T(K_1)^{-1}(\mathbf{U}\mathbf{A}_\mathbf{N}^{-1}).
    \end{split}
    \end{equation*}
    Then, for the Dirichlet Schur complement, we have:
    \begin{equation*}
         S_D = \mathbf{B}\mathbf{A}_\mathbf{D}^{-1}\mathbf{B}^T = \mathbf{B}\mathbf{A}_\mathbf{N}^{-1}\mathbf{B}^T - (\mathbf{U}\mathbf{A}_\mathbf{N}^{-1}\mathbf{B}^T)^T(K_1)^{-1}(\mathbf{U}\mathbf{A}_\mathbf{N}^{-1}\mathbf{B}^T),
    \end{equation*}
    which completes the proof of \eqref{sd_lr}, taking into account that: 
    \begin{equation*}
        \mathbf{A}_\mathbf{N}^{-1}\mathbf{B}^T = ((\mathbf{B}^T\mathbf{B})^\dagger + (\mathbf{C}^T\mathbf{C})^\dagger)\mathbf{B}^T = (\mathbf{B}^T\mathbf{B})^\dagger\mathbf{B}^T = \mathbf{B}^{\dagger}.
    \end{equation*}
    In order to proof \eqref{sd_lr_inv}, we exploit the Sherman–Morrison–Woodbury formula once again for the inverse of $S_D$ under the representation \eqref{sd_lr}.
\end{proof}

\section{Limiting case}\label{sec:limit}
Let us consider the limiting case, when the perturbation matrix $ {I^\mathbf{u}_{\sim}}$ defined in \eqref{perturbation_matrix} has full rank and coincides with the vector velocity identity matrix ${I^\mathbf{u}}$ defined in \eqref{velocity_identity}:
\begin{equation*}
    {I^\mathbf{u}_{\sim}} = {I^\mathbf{u}}.
\end{equation*}
Then, the inverse of the Dirichlet Schur complement matrix defined in \eqref{sd_lr_inv} is reduced as follows:
\begin{equation}
    S_D^{\dagger} = S_N^{\dagger} + \dfrac{2}{h^2}(\mathbf{B}\mathbf{B}^T)^\dagger,
\end{equation}
where the matrix $(\mathbf{B}\mathbf{B}^T):\mathcal{P}^h \to \mathcal{P}^h$ is the pressure Laplacian with the Neumann boundary conditions and the constant nullspace \eqref{pressure_constant}.

practical outcomes: This explains why SIMPLE preconditioner and the preconditioner presented by Frank work well for geometries with high s-t-v ratio.

\bibliographystyle{plain}      
\bibliography{main.bib}   

\end{document}